\documentclass{amsart}
\usepackage[active]{srcltx}
\usepackage{amssymb,color}
\usepackage{float}
\restylefloat{table}

\usepackage[margin=1.25in,dvips]{geometry}

\usepackage{array}
\usepackage[active]{srcltx}
\usepackage{epsfig}
\usepackage{hyperref}
\hypersetup{%
	colorlinks = true,
	linkcolor  = black
}
\usepackage{amsmath}
\usepackage{amssymb}
\usepackage{hyperref}
\usepackage{enumerate}

\usepackage{slashed}

\usepackage{enumitem}

\newtheorem{Proposition}{Proposition}[section]
  \newtheorem{Remark}{Remark}[section]
  
  \newtheorem{Lemma}{Lemma}[section]
  
  \newtheorem{Theorem}{Theorem}[section]
 
 \newtheorem{Definition}{Definition}[section]

\def\CC{\mathbb{C}}

 \def\RR{\mathbb{R}}

\def\Re{\mathrm{Re\,}}

\numberwithin{equation}{section}

\usepackage{xcolor}

\def\be{\begin{equation}}
\def\bel{\begin{equation}\label}
\def\ee{\end{equation}}

\def\bd{\begin{Definition}}
\def\ed{\end{Definition}}

\def\bp{\begin{Proposition}}
\def\bpl{\begin{Proposition}\label}
\def\ep{\end{Proposition}}
\def\bl{\begin{Lemma}}
\def\el{\end{Lemma}}
\def\bt{\begin{Theorem}}
\def\et{\end{Theorem}}

\title{Parabolic cylinder functions revisited using the Laplace transform}
\author{Rodica D. Costin, Georgios Mavrogiannis}

\address{Department of Mathematics, The Ohio State University, Columbus, OH 43210}

\email{costin.10@osu.edu}

\address{Department of Mathematics, Rutgers University, New Brunswick, NJ 08903 USA.}
\email{gm758@math.rutgers.edu}

\date\today

\begin{document}

\begin{abstract}
In this paper we gather and extend classical results for parabolic cylinder functions, namely solutions of the Weber differential equations, using a systematic approach by Borel--Laplace methods.

We revisit the definition and construction of the standard solutions~$U,V$ of the Weber differential equation
\begin{equation*}
w''(z)-\left(\frac{z^2}{4}+a\right)w(z)=0
\end{equation*}
and provide representations by Laplace integrals extended to include all values of the complex parameter~$a$; we find an integral integral representation for the function $V$; none was previously available.

For the Weber equation in  the  form
\begin{equation*}
	u''(x)+\left(\frac{x^2}{4}-a\right)u(x)=0, 
\end{equation*}
we define a new fundamental system~$E_\pm$ which is analytic in~$a\in\mathbb{C}$, based on asymptotic behavior; they appropriately extend and modify the classical solutions $E,E^*$ of the real Weber equation to the complex domain.

The techniques used are general and we include  details and motivations for the approach.
\end{abstract}

\maketitle
{
	\hypersetup{linkcolor=black}
	\tableofcontents
}

\section{Introduction}\label{Intro}

In the present paper we revisit and extend classical results for the Weber differential equation
\bel{otherpcf}
w''(z)-\left(\frac{z^2}{4}+a\right)w(z)=0,
\ee
which is also presented in  the form
\bel{mypcf}
u''(x)+\left(\frac{x^2}{4}-a\right)u(x)=0,
\ee
This is a linear, second order differential equation that admits two turning points. The solutions are called {\em parabolic cylinder functions}.

The two equations~\eqref{otherpcf},~\eqref{mypcf} are related by a linear change of variables. Indeed, denoting by $w(a,z)$ a solution\footnote{The notation for the parabolic cylinder functions in \cite{nist,AS} use the parameter as a first argument; this convention is preserved here. The notation in \cite{Miller}, \cite{Whittaker} uses the parameter  as an index, and in a different presentation.} of \eqref{otherpcf} then 
\begin{equation}\label{changevar}
	u(a,x):=w\left( e^{\frac{2k+1}2\pi i}a,e^{-\frac{2k+1}4\pi i}x\right)
\end{equation}
solves \eqref {mypcf} for any integer $k$. In particular note that $w\left(ia,xe^{-\pi i/4}\right)$ and $w\left(-ia,xe^{\pi i/4}\right)$ solve \eqref{mypcf}. Weber's equation has many applications which are usually in the real domain,  motivating the two presentations of this equation.

In view of the applications the Weber differential equations, the fundamental systems of solutions,  denoted by $U(a,z),\ V(a,z)$ for ~\eqref{otherpcf},  respectively  $E(a,x),\ E^*(a,x)$ for~\eqref{mypcf},
have been studied extensively, see for example~\cite{nist,AS,Miller,OlverUniAsymExpWeber,Temme,Whittaker} and references therein.

\subsection{Our motivation and results}

There is a wealth of information on the parabolic cylinder functions; however, they are often formulated for the real domain (though the equations, and solutions, are entire functions). Several integral representations are valid only for restricted values of the parameter~$a$. Some results are formulated only for real $a$, with a note that `similar formulas should hold in the complex domain', see the discussion in Olver's papers~\cite{OlverUniAsymExpWeber,OlverTurningPoints}. See also~\cite{Temme} for a treatment of equation~\eqref{otherpcf} on the complex plane.

The incompleteness of the information prompted us to revisit the topic, having in mind applications in the complex domain, in the context of normally hyperbolic trapping. This phaenomenon been studied extensively in the literature and an exhaustive list of citations is not possible, see for example~\cite{zworski,bony,DR2,Knoll,Ramond,barreto} and references therein. An example of normally hyperbolic trapping is apparent in the study of the wave equation on the Kerr spacetime, see~\cite{DR2}, which note corresponds to trapping of photons in the black hole exterior. The equations which admit normally hyperbolic trapping take roughly the form~\eqref{otherpcf} or \eqref{mypcf} after appropriate rescalings.

 We use a modern, systematic approach, namely Borel-Laplace techniques.   We obtain integral representation formulas valid for all values of the parameter~$a\in\mathbb{C}$.  We obtain an integral formula for the classical solution~$V$ where none was previously available, see Theorem~\ref{theorem1}. We adapt the formulas for the $E,E^*$ for usage in the complex plane and define them based on general principles.

\begin{Remark}
	The methods used here are general and are applicable to generic differential equations, as detailed in~\cite{OCostin-2009-Book}, starting with ideas pioneered by \'Ecalle in his fundamental work~\cite{Ecalle}. These methods are not difficult to use for the present linear equations, thus this work is also a good introduction to the Borel-Laplace approach. 
\end{Remark}

\subsection{The organization of the paper} 
Our main results are gathered in Theorems~\ref{theorem1},~\ref{theorem2}, preceded by brief expositions of relevant  classical results in Sections\,\ref{classical}, respectively\,\ref{clasic2}.  The proofs of the theorems are found in Sections~\ref{sec: PfTh1}, respectively~\ref{PfTh2}.

The rest of the present section contains explanations of a few technical notions in the simple contexts used here: Section\,\ref{modern-1} contains a brief presentation of the `Hadamard finite part';  Section\,\ref{modern0} contains the definition of the \'Ecalle medianization (M) in the case needed here. Section\,\ref{ALemma} contains the proof of a regularization lemma we use.

\subsection{The Hadamard finite part}\label{modern-1}

For an introduction to the `Hadamard finite part' for integrals of the form~$\int_0^x t^{\alpha-1}f(t)dt$ see~\cite{Costin-2010-orthogonality,Hadamard}. 
For the convenience of the reader we review the concepts and properties we use.

For a function $f$ analytic in a disk 
$D_r=\{x\in\mathbb{C}:|x|<r\}$, continuous up to the boundary, with Maclaurin series~$f(x)=\sum_{n\geq 0} f_n x^n$, the Hadamard finite part of the integral $\int_0^x t^{\alpha-1}f(t)dt$ is defined as
$$\slashed\int_0^x t^{\alpha-1}f(t)dt:=\sum_{n=0}^\infty\frac{f_n}{n+\alpha}x^n$$
which is an analytic function of $x$ on $D_r$ for any $\alpha\in\CC\setminus\{0,-1,-2,\ldots \}$. If $\Re\alpha>0$ this equals the usual integral, and for other complex $\alpha$ it is its analytic continuation, having simple poles when $\alpha=0,-1,-2,\ldots$.

The Hadamard finite part has the usual properties of integrals (except for positivity).

\subsection{The \'Ecalle medianization}\label{modern0} 

For a function~$g$ which is analytic in a sector in $\CC$ containing $\RR_+$ except for at most one point in $\RR_+$ where it may have a singularity (a pole or algebraic/logarithmic branch point) and $g$ decays towards infinity, the \'Ecalle medianization (M) is defined as, see~\cite{Ecalle},
\begin{equation}\label{eq: modern0, eq 1} 
	M\int_0^\infty g(p)\, dp:=\frac12 \int_{e^{i0-}\RR_+}g(p)\, dp+\frac12 \int_{e^{i0+}\RR_+}g(p)\, dp. 
\end{equation}	
where the paths of integration on the right hand side of~\eqref{eq: modern0, eq 1}  are half-lines just below $\RR_+$
and just above~$\RR_+$, respectively.\footnote{For functions $g$ with more than one singularity on $\RR_+$ the \'Ecalle's medianization is a weighted average of integrals along paths winding among these singularities in prescribed ways.
}

Note that if $g$ has no singularity along $\RR_+$ (and decays sufficiently fast  at $\infty$) then the medianization average equals $\int_0^\infty g(p)\, dp$.

\subsection{A lemma for removing singularities}\label{ALemma}

The following Lemma is used in formulas~\eqref{vplus},~\eqref{solvpm} to prove that the singularities of appropriate Laplace type integrals are removable.

\begin{Lemma}\label{uminanalytic}
	
	Let $\alpha\in\CC$ and $\phi(p)$ be a function analytic on $[0,+\infty)$ so that,  for any $R>0$, the integral 
	\begin{equation}
		\int_R^\infty p^{\alpha-1}\phi(p)\,dp
	\end{equation}
	converges for all $\alpha\in\CC$. Then:
	
	\begin{enumerate}
		\item For any $k=0,1,2,\ldots$ we have
		\bel{lim1}
		\lim_{\epsilon\to 0}\ \epsilon\int_0^\infty p^{-k-1+\epsilon}\phi(p)\, dp=\frac{\phi^{(k)}(0)}{k!}
		\ee
where the integral is taken in the sense of its Hadamard finite part, see Section~\ref{modern-1}. 
		
		\item  For $k=0,1,2,\ldots$ we have
		\bel{lim2}
		\lim_{\alpha\to -k}\ \frac1{\Gamma(\alpha)}\int_0^\infty p^{\alpha-1}\phi(p)\, dp=   (-1)^k\phi^{(k)}(0)
		\ee
		As a consequence, the function 
		\begin{equation}\label{eqg1}
			g(\alpha)=\frac1{\Gamma(\alpha)}\int_0^\infty p^{\alpha-1}\,\phi(p)\,dp
		\end{equation}
		can be analytically continued to an entire function. Note that in the formula \eqref{eqg1} the integral sign denotes a Lebesgue integral for $\Re\alpha>0$, and its Hadamard finite part for $\Re\alpha\le 0,\ \alpha\ne 0,-1,-2,\ldots$. \\
		
	\end{enumerate}

\end{Lemma}

\begin{proof} 	
	
	We begin with~(A). We write
	\begin{equation}
		\phi(p)=\sum_{n=0}^ka_np^n+p^{k+1}\tilde{\phi}(p)
	\end{equation}
	where $\tilde{\phi}(p)$ is analytic on $[0,+\infty)$. Then, using the properties of the Hadamard finite part,
	we have the following
	\begin{equation}
		\begin{aligned}
			\int_0^\infty p^{-k-1+\epsilon}\,\phi(p)\, dp &	=\int_0^R p^{-k-1+\epsilon}\,\phi(p)\, dp+\int_R^\infty p^{-k-1+\epsilon}\,\phi(p)\, dp\\
			&	=\int_0^R \sum_{n=0}^ka_np^{n-k-1+\epsilon}\, dp+\int_0^R\ p^{\epsilon}\,\tilde{\phi}(p)\, dp +\int_R^\infty p^{-k-1+\epsilon}\,\phi(p)\, dp\\
			&	=\sum_{n=0}^k\frac{a_n}{n-k+\epsilon}\,  R^{n-k+\epsilon}+\int_0^R\ p^{\epsilon}\,\tilde{\phi}(p)\, dp+\int_R^\infty p^{-k-1+\epsilon}\,\phi(p)\, dp,
		\end{aligned}
	\end{equation}
	for any $R>0$. Therefore, we obtain
	\begin{equation}
		\lim_{\epsilon\to 0}\ \epsilon\int_0^\infty p^{-k-1+\epsilon}\phi(p)\, dp=a_k
	\end{equation}
	and  \eqref{lim1} follows.

	We continue with~(B). We use the reflexion formula $\Gamma(\alpha)\Gamma(1-\alpha)=\frac{\pi}{\sin(\pi \alpha)}$ and denote $\alpha=-k+\epsilon$. We note that then~\eqref{lim2} follows from \eqref{lim1}.
	
\end{proof}

\section{Main Theorems}\label{sec: theorems}

The present Section is organized as follows: In Sections~\ref{classical},~\ref{clasic2} we discuss classical results for the equations~\eqref{otherpcf},~\eqref{mypcf} respectively 
and in Sections~\ref{modern},~\ref{modern2} we present our results for the two differential equations. We emphasize that in the proofs of our results, contained in Sections~\ref{sec: PfTh1},~\ref{PfTh2}, we use Borel--Laplace methods,  where  we provide details and motivations of the constructions.

\subsection{Classical results on Weber's equation~\eqref{otherpcf}}\label{classical}  

The results presented in this section are found in the classical references \cite{nist,AS}. 

Two linearly independent solutions $U(a,z),V(a,z)$ of \eqref{otherpcf} are defined in terms of even and odd solutions\footnote{In \cite{Whittaker} the function $U$ is denoted by $D_n(z)$, where the two notations are connected by $D_n(z)=U\left(-n-\tfrac 12,z\right)$; its theory is developed in a way closer to our approach in \S\ref{modern}.}; they have the asymptotic behavior
\bel{asymU}
U\left(a,z\right)\sim e^{-\frac{1}{4}z^{2}}z^{-a-\frac{1}{2}}\sum_{s=0}^{%
	\infty}(-1)^{s}\frac{{\left(\frac{1}{2}+a\right)_{2s}}}{s!(2z^{2})^{s}},\ \ \ \ \ \ \text{as } z\to\infty,\ |\arg z|<\frac{3\pi}4
\ee
\bel{asymV}
V\left(a,z\right)\sim\sqrt{\frac{2}{\pi}}e^{\frac{1}{4}z^{2}}z^{a-\frac{1}{2}}%
\sum_{s=0}^{\infty}\frac{{\left(\frac{1}{2}-a\right)_{2s}}}{s!(2z^{2})^{s}}, \ \ \ \ \text{as } z\to\infty,\  |\arg z|<\frac{\pi}4.
\ee
see~[\cite{nist},12.9.1-2], [\cite{nist},12.5.2.]

The solution $U(a,z)$ has a number of integral representations, for example
\begin{equation}
	\label{solutionUfromAS}
	U\left(a,z\right)=\frac{ze^{-\frac{1}{4}z^{2}}}{\Gamma\left(\frac{1}{4}+\frac{%
			1}{2}a\right)}\*\int_{0}^{\infty}t^{\frac{1}{2}a-\frac{3}{4}}e^{-t}\left(z^{2}%
	+2t\right)^{-\frac{1}{2}a-\frac{3}{4}}\,\mathrm{d}t,\ \ \ \  \ \ \text{for }|\arg z|<\frac{\pi}2,\ \ \Re a>-\tfrac12
	\ee
	see~\cite{AS}, or alternatively~[\cite{nist}, 12.5.2]. There were no integral formulas for the function $V$ previously available; see our formula in Theorem~\ref{theorem1} for such an integral representation.

	Finally, the following connection formulas hold
	\begin{equation}\label{connectionUV}
		\begin{aligned}
			U\left(a,-z\right)&=-\sin\left(\pi a\right)U\left(a,z\right)+\frac{\pi}{\Gamma
				\left(\frac{1}{2}+a\right)}V\left(a,z\right),\\
				V\left(a,-z\right)&=\frac{\cos\left(\pi a\right)}{\Gamma\left(\frac{1}{2}-a%
					\right)}U\left(a,z\right)+\sin\left(\pi a\right)V\left(a,z\right),
		\end{aligned}
	\end{equation}
see [\cite{nist},12.2.15-16], and \cite{Temme}.

\subsection{The main Theorem~\ref{theorem1}: A Laplace approach to Weber's equation~\eqref{otherpcf}}\label{modern} 

From the point of view of modern Borel-Laplace analysis it is natural to define a fundamental system of solutions based on their asymptotic behavior. This is done in Theorem\,\ref{theorem1} where these functions are presented as  Laplace integrals: these are the Borel sums of their asymptotic series \eqref{asymU}, \eqref{asymV}.

\begin{Theorem}\label{theorem1}
	
	Let~$a\in\mathbb{C}$. The Weber equation~\eqref{otherpcf} admits two linearly independent solutions~$U,V$,entire in~$a$ and $z$, with the following Laplace type integral representation formulas 
	\begin{equation}\label{Visup}
		U(a,z)=u_-(a,z^2),\qquad V(a,z)=\sqrt{\frac{2}{\pi}}\, {u}_+(a,z^2),
	\end{equation}
	where 
	\bel{eq: theorem1, eq 1}
		\begin{aligned}
			u_-(a,y)	&	=\frac{ 1}{\Gamma\left(\frac{1}{4}+\frac{1}{2}a\right)}e^{-y/4}\int_{0}^{\infty}\, e^{-py}\, p^{\frac{1}{2}a-\frac{3}{4}}\,\left(1+2p\right)^{-\frac{1}{2}a-\frac{3}{4}}\,\mathrm{d}p,\ \ \ \ (\Re y>0)\\
		\end{aligned}
	\end{equation}	
	\bel{eq: theorem1, eq 2}
		\begin{aligned}
			u_+(a,y) 	&	= \frac{e^{y/4}}{\Gamma \! \left(\frac{1}{4}-\frac{a}{2}\right)} \, M\int_0^\infty  e^{-py} p^{-\frac{3}{4}-\frac{a}{2}} \left(1-2 p \right)^{-\frac{3}{4}+\frac{a}{2}}
			\, dp,\ \ \ \ (\Re y>0)
		\end{aligned}
	\end{equation}
where for the \'Ecalle medianization~$M$ see Section~\ref{modern0}. The integrals~\eqref{eq: theorem1, eq 1},~\eqref{eq: theorem1, eq 2}, if divergent at $0$ (that is, for~$\Re (a)\leq -\frac{1}{2}$ and~$\Re (a)\geq \frac{1}{2}$ respectively), must be understood as their Hadamard finite part, see Section~\ref{modern-1}. Moreover, the singularities 
\begin{equation*}
	\tfrac14+\tfrac 12 a=0,-1,-2,...,\qquad \text{respectively }\ \ 	\tfrac14-\tfrac 12 a=0,-1,-2,...,
\end{equation*}
of the representation formulas~\eqref{eq: theorem1, eq 1},~\eqref{eq: theorem1, eq 2} respectively, are removable.

\end{Theorem}

The proof of Theorem~\ref{theorem1} is found in Section~\ref{sec: PfTh1}. We emphasize that the result of Theorem~\ref{theorem1} holds for all~$a\in \mathbb{C}$. The asymptotic series \eqref{asymU},  \eqref{asymV} are a mere consequence of \eqref{eq: theorem1, eq 1}, \eqref{eq: theorem1, eq 2} using Watson's Lemma.
In the proof we calculate the connection formulas for $u_+,u_-$ which prove that $V$ is a multiple of $u_+$, see \eqref{Visup} and re-prove \eqref{connectionUV}.
The connection is obtained by analytic continuation of \eqref{eq: theorem1, eq 1}, \eqref{eq: theorem1, eq 2}, achieved by rotating the path of integration of the Laplace integrals in the complex plane, whose calculation we detail; notably, this technique unravels the Stokes phenomena.

\subsection{Classical results on Weber's equations~\eqref{mypcf}}\label{clasic2}
For the equation in form ~\eqref{mypcf} more often than not results are presented for real parameter, followed by a quick remark that the results for complex parameter generalize appropriately; see for example the discussion in Olver's~\cite{OlverTurningPoints}. The `complex solutions' $E,E^*$  are indirectly defined and contain a factor which is not analytic in $a$, as seen below.

The classical reference~[\cite{AS},19.17.7]  defines two linearly independent solutions $E,E^*$ of \eqref{mypcf} in terms of Whittaker solutions
\bel{defE}
\begin{aligned}
	E(a,x) & =k^{-1/2}W(a,x)+ik^{1/2}W(a,-x),\\
	E^*(a,x) & =k^{-1/2}W(a,x)-ik^{1/2}W(a,-x)
\end{aligned}
\ee

The Whittaker functions are defined as solutions determined by their initial values, see [\cite{nist} 12.14.1-2]. The following hold:
\bel{WdeU}
\begin{aligned}
	W\left(a,x\right)	& =\sqrt{k/2}\,e^{\frac{1}{4}\pi a}\left(e^{i\rho}U\left(ia,xe^%
{-\pi i/4}\right)+e^{-i\rho}U\left(-ia,xe^{\pi i/4}\right)\right),\\
W\left(a,-x\right)	&	=\frac{-i}{\sqrt{2k}}\,e^{\frac{1}{4}\pi a}\left(e^{i\rho}U%
\left(ia,xe^{-\pi i/4}\right)-e^{-i\rho}U\left(-ia,xe^{\pi i/4}\right)\right)
\end{aligned}
\ee
with
\begin{equation}\label{eq: WdeU, eq 1}
k=\sqrt{1+e^{2\pi a}}-e^{\pi a},\qquad \rho=\tfrac{1}{8}\pi+\tfrac{1}{2}\phi_{2},\qquad \phi_{2}=\operatorname{ph}\Gamma\left(\tfrac{1}{2}+ia\right). 
\end{equation}

Therefore
\bel{EU}
\begin{aligned}
	E(a,x) & =\sqrt{2}\,e^{\frac{1}{4}\pi a}\, e^{i\rho}U\left(ia,xe^{-\pi i/4}\right),\\
	E^*(a,x)	&	=\sqrt{2}\,e^{\frac{1}{4}\pi a}\, e^{-i\rho}U\left(-ia,xe^{\pi i/4}\right)
\end{aligned}
\ee
see also [\cite{AS}, 19.17.9].

 
 From~\eqref{defE} and \eqref{WdeU} we obtain the connection formulas 
 \bel{connectionE}
\begin{aligned}
	 E(a,-x)	& =-ie^{\pi a}E(a,x)+i \sqrt{1+e^{2\pi a}}E^*(a,x),\\
	  E^*(a,-x)	&  =-i \sqrt{1+e^{2\pi a}}E(a,x)+ie^{\pi a}E^*(a,x)
\end{aligned}
\ee
where for the second formula see also [\cite{AS}, 19.18.3]

\subsection{The main Theorem~\ref{theorem2}: A Laplace approach to Weber's equation~\eqref{mypcf}}\label{modern2}

We note that the angle $\phi_2$ used in these classical references introduces a factor not analytic in $a$ in the definitions of $E,E^*$, see \eqref{eq: WdeU, eq 1}. It is natural to expect that the factor $e^{i\phi_2}$ appearing in the formulas above could be replaced by $\Gamma(\tfrac12+ia)\sqrt{\cosh \pi a}/\sqrt{\pi}$, thus restoring analyticity in $a$. In Theorem\,\ref{theorem2} below we show that this is indeed the case. Furthermore, we introduce a related fundamental system $E_\pm$, based on asymptotic behavior at infinity which, importantly, is entire in $a$.

\begin{Theorem}\label{theorem2}
	
	Let~$a\in\mathbb{C}$. The Weber equation~\eqref{mypcf} admits two linearly independent solutions~$E_\pm$, entire in $a$ and $x$, with the following Laplace type integral representations formulas
\bel{defEpm}
E_\pm(a,x):=v_\pm(x^2)
\ee
where
\bel{solvp}
v_+(s)  = \frac1{\Gamma\left(\frac{ia}2+\frac14\right)}\,e^{is/4} \int_0^\infty\, e^{-ps} p^{\frac{ia}2-\frac34}\, (1+2ip)^{-\frac{ia}2-\frac34}\, dp, \ \ \  ( \Re s>0)
\ee
\bel{solvm}
v_-(s)=\frac1{\Gamma\left(-\frac{ia}2+\frac14\right)}\, e^{-is/4}\,\int_0^\infty\, e^{-ps} p^{-\frac{ia}2-\frac34}\, (1-2ip)^{\frac{ia}2-\frac34}\, dp,\ \ \ \  ( \Re s>0).
\ee
The integrals~\eqref{solvp},~\eqref{solvm}, if divergent at $0$ (namely for~$Im (a)\geq \frac{1}{2}$ and~$Im (a)\leq -\frac{1}{2}$ respectively), must be understood as their Hadamard finite part, see Section~\ref{modern-1}. Moreover, the singularities
\begin{equation*}
	\tfrac14+\tfrac i2 a=0,-1,-2,...,\qquad\text{repectively } \ \   \ \tfrac14-\tfrac i2 a=0,-1,-2,..
\end{equation*}
of the representation formulas~\eqref{solvp},~\eqref{solvm} respectively, are removable.

The solutions $E_\pm$ have the following asymptotic behavior
\begin{equation}\label{asymEpmTh}	
E_\pm(a,x)\sim e^{\pm ix^2/4} (x^2)^{\mp\frac{ia}{2}-\frac{1}{4}}\left(1+\mathcal{O}(x^{-2})\right),\qquad  \text{as}~x\rightarrow \infty	,\ \ |\arg x|<\frac{\pi}4
\end{equation}

The solutions $E_\pm$  satisfy the connection formulas 
\begin{equation}\label{conEm}
	E_-(a,-x)=ie^{\pi a} E_-(a,x)+ \frac{\sqrt{2\pi}}{\Gamma\left(\frac12-ia\right)} e^{\pi a/2}\,  E_+(a,x)
\end{equation}
and
\begin{equation}\label{conEp}
E_+(a,-x)=-ie^{\pi a} E_+(a,x)+ \frac{\sqrt{2\pi}}{\Gamma\left(\frac12+ia\right)} e^{\pi a/2}\,  E_-(a,x). 
\end{equation}

Finally, the functions $E_\pm$ are related to the classical pair $E,E^*$ in Section~\ref{clasic2}, as follows
\begin{equation}\label{EEstar}
	E(a,x)=\sqrt{2}e^{\frac{\pi i}4+\frac{i\phi_2}2}E_+(a,x),\ \ \ E^*(a,x)=\sqrt{2}e^{-\frac{\pi i}4-\frac{i\phi_2}2}E_-(a,x),
\end{equation}
where for~$\phi_2$ see~\eqref{eq: WdeU, eq 1}. 
\end{Theorem}

The proof of Theorem\,\ref{theorem2} is found in Section~\ref{PfTh2}. We use Borel-Laplace techniques to construct a fundamental system of (Jost) solutions $E_\pm(a,x)$, entire in $a$ based on asymptotic behavior at $+\infty$:  the integral formulas~\eqref{solvp},~\eqref{solvm}  represent the Borel sum of their respective asymptotic series. It turns out that our new fundamental system is an explicit multiple of $E,E^*$. We prove the connection formulas for the new fundamental system~$E_\pm$ in Proposition~\ref{P1} based on rotation of the path of integration in the complex plane. We note that since solutions of \eqref{mypcf} can be obtained from solutions of  \eqref{otherpcf} using \eqref{changevar},
 it is clear that the results of Section~\ref{modern} can be transcribed for equation \eqref{mypcf}. It is however worth re-working the construction, since this case appears simpler (due to the position of the singularities in the Borel plane).

\begin{Remark} Note that for the values of the parameter $a$ for which the $\Gamma$ functions have poles in \eqref{conEm}, \eqref{conEp} respectively, there exist solutions of the Weber equation~\eqref{mypcf} with the same asymptotic behavior at $+\infty$ and $-\infty$ along the real line. (In fact, it turns out that for~$a=\pm  \left(n+\tfrac12\right) i,\ n=0,1,2,\ldots$ the solution $E_\pm$ has the form $e^{\pm ix^2/4}p_n(x)$, with $p_n$ polynomial of degree $n$, see for example~\cite{nist}.)
		
In particular for
	\begin{equation}
		a=\left(n+\frac{1}{2}\right)i,\qquad n\geq 0
	\end{equation}
	we have that~$E_+(a,x)$ is an entire solution of the Weber equation~\eqref{mypcf} with the same asymptotic behavior at both ends of the real line:
	\begin{equation}
		e^{+ ix^2/4} \,(x^2)^{- \frac{ia}2-\frac14}\left(1+O(x^{-2})\right),\qquad x\to \pm\infty
	\end{equation}
	This is associated with the resonances found in~\cite{zworski} with microlocal methods.
\end{Remark}

\section{Proof of Theorem~\ref{theorem1}}\label{sec: PfTh1}

\subsection{Normalization}\label{subsec: sec: PfTh1, subsec1}

In order to use Borel-Laplace methods, the first step is to normalize the equation.

The proper variable in which a Laplace transform is most useful is the \`Ecalle critical time. Its form is can be found by a classical WKB calculation for equation~\eqref{otherpcf}, which gives solutions with asymptotic behavior
\bel{asymv}
w(z)\sim  e^{\pm\frac{1}{4}z^{2}}z^{\pm a-\frac{1}{2}} \left(1+O(z^{-2})\right),\ \ \ \ \text{as }z\to+\infty
\ee
up to multiplicative constants. The form of the exponential term in \eqref{asymv} suggests that the \'Ecalle critical time\footnote{This is a change of variable after which solutions are Borel transformable. If $f(x)=\int_0^\infty e^{-px}F(p)dp$ we say that $f$ is the Laplace transform of $F$, and $F$ is the inverse Laplace trasform, or Borel transform, of $f$.} is $z^2$. Passing to the new variable $y=z^2$, and with $w(z)=u(z^2)$ the Weber equation~\eqref{otherpcf} becomes 
\bel{equ}
y\,u''(y)+\frac12 u'(y)+\left(\frac y{16}-\frac a4\right)u(y)=0.
\ee

\subsection{The proof of~\eqref{eq: theorem1, eq 1}: The function~\texorpdfstring{$U(a,z)$}{g}}\label{subsec: sec: PfTh1, subsec2}
 
 In this subsection we find the function~$U(a,z)$ defined as the solution decaying as the small exponential $e^{-z^2/4}$ as $z\to+\infty$.

We substitute
 \begin{equation}\label{utof}
 	u(y)=e^{-y/4}f(y)
 \end{equation}
  in~\eqref{equ} and then express $f$ as a Laplace transform 
 \begin{equation}
 f(y)=\int_0^\infty e^{-py}F(p)\,dp,
 \end{equation}
  It follows that $F$ must satisfy
\begin{equation}\label{eqFp}
p \left(2 p +1\right)\, F' \left(p \right)+\frac{ 3-2 a +12 p }{4}\, F \! \left(p \right)=0,
\end{equation}
having as solutions 
\bel{formF}
F(p) = p^{\frac{1}{2}a-\frac{3}{4}}\,\left(1+2p\right)^{-\frac{1}{2}a-\frac{3}{4}},
\ee
and any constant multiple of the above; we choose a suitable constant later, in~\eqref{defuminus}.

\begin{Remark}\label{Rem0} 
We used the substitution \eqref{utof} since it produces an equation for $F$ simple to solve in closed form; the downside is that $F$ is not analytic at $0$ (and it is not defined for certain values of $a$). The asymptotic behavior \eqref{asymv} suggests that the `correct' substitution would be 
	\begin{equation}
		u(y)=e^{-y/4}y^{- \frac a2+\frac{3}{4}}\phi(y)
	\end{equation}
	and now the function $\phi$  is the Laplace transform of a function analytic at $0$ (it is expressible in terms of hypergeometric functions). The two approaches are, of course, equivalent.
\end{Remark}

Therefore, equation~\eqref{equ} has the solution
\bel{vminus}
\tilde{u}_-(y)=e^{-y/4}\int_{0}^{\infty}\, e^{-py}\, p^{\frac{1}{2}a-\frac{3}{4}}\,\left(1+2p\right)^{-\frac{1}{2}a-\frac{3}{4}}\,\mathrm{d}p,\quad\text{for}~\Re y>0
\ee
up to multiplicative constants.

The integral \eqref{vminus} is defined in the classical sense if $\Re (a)>-\tfrac12$. For other values of $a$ the integral is understood as its Hadamard finite part, see Section~\ref{modern-1}.

We choose our solution to be the following constant multiple of $\tilde{u}_-(y)$: let
\bel{defuminus}
u_-(a,y):=\frac{ 1}{\Gamma\left(\frac{1}{4}+\frac{1}{2}a\right)}\ \tilde{u}_-(y).
\ee
and this choice ensures, by Watson's Lemma, that $u_-(a,z^2)$ has the asymptotic behavior \eqref{asymv}, with multiplicative constant $1$. 

Now, the Hadamard finite part of~\eqref{vminus} is not defined when
$a= - \left(2n+\frac{1}{2}\right),\qquad n=0,1,2,\ldots$ 
where it has a first order pole. But so does the prefactor in \eqref{defuminus}, and in view of Lemma~\ref{uminanalytic} these are removable singularities of $u_-$.

We obtain the Laplace integral representation~\eqref{eq: theorem1, eq 1} for~$u_-$ by combining \eqref{vminus}, \eqref{defuminus}. Note that~\eqref{eq: theorem1, eq 1} is the Borel summation of the asymptotic series \eqref{asymU} with $z^2=y$. Finally, note that 
\bel{vmU}
U(a,z)=u_-(a,z^2),
\ee
since both $u_-(a,z^2)$ and $U(a,z)$ are solutions of \eqref{otherpcf} with the same, exponentially small, asymptotic behavior.

\subsection{The proof of~\eqref{eq: theorem1, eq 2}: The function~\texorpdfstring{$V(a,z)$}{g}}\label{subsec: sec: PfTh1, subsec3}

A second, independent, solution of equation~\eqref{otherpcf} must have the large exponential behavior $e^{z^2/4}$ as $z\to+\infty$. This asymptotic property does not characterize a unique solution, since to such a solution one can add any multiple of the small solution $U$, obtaining another solution with the same (exponentially increasing) asymptotic series. We will proceed with the construction, similarly to that in Section~\ref{subsec: sec: PfTh1, subsec2} for~$U$, and we will see how this lack of uniqueness is observed in the Borel $p$-plane and how a solution is chosen using \'Ecalle's medianization. (This choice is called the Borel sum of the asymptotic series \eqref{asymV}, and it 
ensures that the correspondence provided by this generalized Laplace transform between formal, asymptotic, solutions and actual solutions commutes with all operations \cite{Ecalle}, \cite{OCostin-2009-Book}. See also \eqref{vpV} below and footnote \ref{balancedchoice}.)

Similarly to Section~\ref{subsec: sec: PfTh1, subsec2},  we now substitute  in~\eqref{equ}
\begin{equation}
	v(y)=e^{y/4}g(y)
\end{equation}
then express $g$ as a Laplace transform 
\begin{equation}
g(y)=\int_0^\infty e^{-qy}G(q)\,dq. 
\end{equation}
and as in Section~\ref{subsec: sec: PfTh1, subsec2} we obtain that
\bel{defG}
G(q)= q^{-\frac{3}{4}-\frac{a}{2}} \left(1-2 q \right)^{-\frac{3}{4}+\frac{a}{2}},
\ee
up to any multiplicative constant; we choose this constant later, in~\eqref{vplus}.

 Note that $G(q)$ is not Laplace transformable on $\RR_+$ due to its branch point at $q=\tfrac12$. Moreover, there is an  ambiguity in choosing the branch: the Laplace transform of $G$ on a ray above $\RR_+$ produces a solution of the equation, and so does the Laplace transform of $G$ on a ray below $\RR_+$ (their difference is another solution, exponentially small). In such situations it is convenient, when  using Borel-Laplace  methods, to use the  \'Ecalle medianization, see Section~\ref{modern0}.

 Namely, we  consider the solution
 \begin{equation}\label{tvplus}
 	\begin{aligned}
 		\tilde{u}_+(y) & =e^{y/4}\, M\int_0^\infty  e^{-qy} q^{-\frac{3}{4}-\frac{a}{2}} \left(1-2 q \right)^{-\frac{3}{4}+\frac{a}{2}}
 		\, dq\\
 		&	:=
 		 e^{y/4}\left(\frac12 \int_{e^{ i 0+}\RR_+} + \frac12 \int_{e^{ i 0-}\RR_+} \right) e^{-qy} q^{-\frac{3}{4}-\frac{a}{2}} \left(1-2 q \right)^{-\frac{3}{4}+\frac{a}{2}}
 		\, dq ,\ \ \ \ ({\rm{ for}}\   \Re y>0)
 	\end{aligned}
 \end{equation}
where $e^{ i 0\pm}\RR_+$ are rays, form $0$ to $\infty$, slightly above/below the real axis respectively.

Similarly to Section~\ref{subsec: sec: PfTh1, subsec2}, the integrals in \eqref{tvplus} are convergent at $q=0$ if $\Re a<\tfrac12$, while for other values of $a$ they are understood as the Hadamard finite part as long as
\begin{equation*}
a\ne 2k+\tfrac12,\qquad  k=0,1,2,\ldots.
\end{equation*}

The branches at the singularity $q=\tfrac12$ are determined by analytic continuation of $G(q)$ assuming that for $q\in(0,\tfrac12)$ we have 
\begin{equation}
	\arg(q)=0,\qquad \arg(1-2q)=0
\end{equation}
For example, for $q$ on the half-lines $\tfrac12+e^{ i 0\pm}\RR_+$ we have $\arg(1-2 q)=\mp\pi$.

As in Section~\ref{subsec: sec: PfTh1, subsec2} we choose the multiplicative constant which ensures the asymptotic behavior~\eqref{asymv} and define
\bel{vplus}
u_+(a,y)=\frac1{\Gamma \! \left(\frac{1}{4}-\frac{a}{2}\right)}\tilde{u}_+(y)
=\frac{e^{y/4}}{\Gamma \! \left(\frac{1}{4}-\frac{a}{2}\right)} \, M\int_0^\infty  e^{-py} p^{-\frac{3}{4}-\frac{a}{2}} \left(1-2 p \right)^{-\frac{3}{4}+\frac{a}{2}}
\, dp,\qquad \Re y>0,
\ee
 Furthermore, the singularities of the formula~\eqref{vplus} at 
\begin{equation*}
	a= 2k+\tfrac12,\qquad k=0,1,\dots 
\end{equation*}
are removable by Lemma\,\ref{uminanalytic}, therefore ${u}_+(a,y)$ is entire in the parameter $a$.

Based solely on their asymptotic behavior, the solution $u_+$ and the classical solutions $U,V$ must be connected by a relation 
 \begin{equation}\label{eqconst}
 V(a,z)=\sqrt{\frac{2}{\pi}}\, {u}_+(a,z^2)+const.\, U(a,z).
 \end{equation}

 In Section~\ref{calccon} below we calculate the connection formula for the solutions $u_+,u_-$. It will turn out that, in view of \eqref{vmU} and \eqref{eqconst}, this connection recovers~\eqref{connectionUV} provided that the constant in \eqref{eqconst} is zero\footnote{In fact the choice of the path of integration as the medianization ensures that this constant, which is `beyond all orders', is zero, see \cite{OCostin-2009-Book}.\ \label{balancedchoice}}, which will imply that
 \bel{vpV}
V(a,z)=\sqrt{\frac{2}{\pi}}\, {u}_+(a,z^2).
\ee

\subsection{Connection formulas and the proof of \eqref{vpV}}\label{calccon}

The connection formulas are calculated by analytically continuing the functions $u_\pm$, defined by Laplace transforms valid for $\Re y>0$, up to $y<0$.
We start by explaining how analytic continuation of functions presented as Laplace integrals can be computed by rotating the path of integration. We then calculate the connection formulas by Laplace continuation methods, which then are compared with \eqref{connectionUV} to deduce \eqref{vpV}.

\subsubsection{\textbf{A brief discussion on `the connection problem'}}\label{subsec: calccon, subsec 1}

 We have constructed 
 \begin{equation}
 u_\pm(a,z^2),
 \end{equation}
 see~\eqref{vplus},~\eqref{eq: theorem1, eq 1}, which are two independent solutions of \eqref{otherpcf}, with asymptotic behavior \eqref{asymv} for $z\to +\infty$. There exists, a possibly different, fundamental pair of solutions with the same behavior as $z\to -\infty$. Obtaining the relation between these two pairs of solutions is called the connection problem. To solve it we take each solution $u_\pm(a,z^2)$ and determine its asymptotic behavior for  $z\to -\infty$ by performing its analytic continuation  in the complex plane from $z\in\RR_+$ to $z\in\RR_-$\footnote{It is apriori known that solutions of \eqref{otherpcf} are entire functions. In generic ODEs solutions are not entire, and by analytically continuing a solution to the boundary of its domain of analyticity we can discover the singularities that the solution develops there \cite{OCostinRCostin-2001-Singularities}.}. We do this by continuing solutions alongs paths going clockwise in $\CC$. (The result is the same to the one obtained by counter-clockwise continuation, since the solutions are entire functions.)

In variable $y=z^2$, this amounts to a clockwise continuation in $y$, from $\arg y=0$, by an angle $2\pi$.

\subsubsection{\textbf{Technique used in continuation}}\label{subsec: calccon, subsec 2} Since~$u_\pm(a,y)$ have representations in terms of Laplace integrals, continuation can be done by rotating the path of integration. The principle is as follows: 
\begin{enumerate}
	\item To continue $u_-$ consider its integral representation~\eqref{eq: theorem1, eq 1} for $\arg y=0$. \\
	\item This integral representation remains the same if $y$ is rotated by an angle $-\theta\in(-\tfrac\pi 2,0)$, which then again remains the same if the path of integration (initially $\RR_+$) is rotated by the (opposite) angle $\theta$, ending up with an integral along $e^{i\theta}\RR_+$. \\
	\item This process is continued (simultaneously rotating $y$ clockwise and $p$ counterclockwise while maintaining that $\Re(py)>0$) until the path hits a singularity of the integrand. \\
	\item To rotate past this singularity, if it is a pole, then a residue must be collected, or, if it is a branch point, then an integral hanging around a cut must be collected, as well as there is change of branch.\footnote{The path of integration lies on the surface covering $\CC\setminus\{0,-\tfrac12\}$. We can follow the values of the integrand on this surface by taking two half-line cuts stemming at $p=0$ and at $p=-\tfrac12$.}
\end{enumerate}

\subsubsection{\textbf{The analytic continuation and connection formulas}}

 Before implementing the strategy discussed above, it is useful to deduce some relations about the solution $u_+$. Specifically, for $G$ as in \eqref{defG}, we write the half-line integrals as follows
\begin{equation}\label{Ipm}
	\begin{aligned}
		I_\pm(y) &	:=\int_{e^{\pm i 0}\RR_+}e^{-qy} G(q)\, dq=\int_0^{\frac12}e^{-qy} G(q)\, dq+\int_{\frac12+e^{\pm i 0}\RR_+} e^{-qy} G(q)\, dq\\
		&	=J(y)+e^{\mp \pi i(\frac a2-\frac34)} \, e^{y/4}\, \int_{\frac12}^\infty e^{-qy}q^{-\frac{3}{4}-\frac{a}{2}} \left|1-2 q \right|^{-\frac{3}{4}+\frac{a}{2}}\, dq\\
		&	=J(y)+2^a\, e^{\mp \pi i(\frac a2-\frac34)} u_-(y),\\
	\end{aligned}
\end{equation}
where 
\begin{equation}
J(y)=\int_0^{\frac12}e^{-qy} G(q)\, dq. 
\end{equation}

Therefore, equation \eqref{tvplus} can be rewritten as
\bel{upJ}
\tilde{u}_+(y)=
J(y)+2^a \, \cos\left(\frac {\pi a}2-\frac{3\pi}4\right)\, \tilde{u}_-(y).\\
\ee
From \eqref{Ipm}, \eqref{upJ} it follows that
\bel{Iminus}
I_-(y)=\tilde{u}_+(y)+ i2^a \, \sin\left(\frac {\pi a}2-\frac{3\pi}4\right)\, \tilde{u}_-(y)
\ee
and
\bel{upIp}
\tilde{u}_+(y)=I_+(y)+i2^a \, \sin\left(\frac {\pi a}2-\frac{3\pi}4\right)\, \tilde{u}_-(y).
\ee

It suffices to find the analytic continuation of $\tilde{u}_-$ (a constant multiple of $u_-$), by applying the strategy of Section~\ref{subsec: calccon, subsec 2}. We analytically continue $\tilde{u}_-(y)$ given by \eqref{vminus} for $\arg y=0$ by simultaneously rotating $y$ clockwise and ray of integration in $p$ counterclockwise. This can be done freely until $\arg p=\pi-0$, just before the ray of integration hits the brach point $p=-\tfrac12$. Continuation past a branch point results in a term consisting of an integral hanging around a cut, as seen below.

To keep track of various branches, assume that initially for $p\in\left(0,\tfrac12\right)$ we have
\begin{equation}
	\arg p=0,\qquad \arg(1-2p)=0,
\end{equation}
and take cuts along 
\begin{equation}
	e^{(2\pi-0)i}\RR_+,\qquad -\tfrac12-\RR_+.
\end{equation}
We denote
\begin{equation}
	\ell_+=e^{i(\pi-0)}\RR_+,\ \ \ \ \ \ell_-=e^{i(-\pi+0)}\RR_+
\end{equation}
and let $F(p)$ be as in \eqref{formF}. Then, for $\Re y<0$ we have
\bel{cont1}
e^{-y/4}\int_{\ell_+}e^{-py}\, F(p)\,\mathrm{d}p=e^{-y/4}\int_{\ell_-}e^{-py}\, F(p)\,\mathrm{d}p+e^{-y/4}\left(\int_{\ell_+} -\int_{\ell_-}\right)e^{-py}\, F(p)\,\mathrm{d}p.
\ee
 Upon continuing the first term on the right hand side of \eqref{cont1} until $\arg y=2\pi$, we obtain
 \begin{equation}
 	e^{2\pi i \left({\frac{1}{2}a-\frac{3}{4}}\right)}\tilde{u}_-(y). 
 \end{equation}

We treat the second term of~\eqref{cont1}, by keeping track of the branches of the factor $\left(1+2p\right)^{-\frac{1}{2}a-\frac{3}{4}}$
in the integrand, and obtain 
\begin{multline}\label{deltal}
\left(\int_{\ell_+} -\int_{\ell_-}\right)e^{-py}\, F(p)\,\mathrm{d}p=\left(\int_{-\tfrac12+\ell_+} -\int_{-\tfrac12+\ell_-}\right)e^{-py}\, F(p)\,\mathrm{d}p\\
=\left( e^{2\pi i \left({-\frac{1}{2}a-\frac{3}{4}}\right)} -1\right)\int_{-\tfrac12+\ell_-}e^{-py}\, F(p)\,\mathrm{d}p,
\end{multline}
where for the last equality note that for
\begin{equation}
	p\in-\tfrac12+\ell_-,\qquad p\in-\tfrac12+\ell_+
\end{equation}
we have respectively
\begin{equation}
	\arg(1+2p)=\pi,\qquad \arg(1+2p)=-\pi.
\end{equation}

The last integral in \eqref{deltal} can be analytically continued up to $\arg y=-0$; changing the variable of integration $p=-\tfrac12+q$ we obtain
\begin{equation}
	e^{-y/4}\int_{-\tfrac12+\ell_-}e^{-py}\, F(p)\,\mathrm{d}p=2^{-a}e^{\pi i\left({\frac{1}{2}a-\frac{3}{4}}\right)}I_-(y),
\end{equation}
where $I_-$ is given by the first equality in \eqref{Ipm}.

All in all we obtain 
\begin{equation}\label{eq:proof of thm1: connection,eq 1}
AC\,\tilde{u}_-(y)=e^{2\pi i \left({\frac{1}{2}a-\frac{3}{4}}\right)}\tilde{u}_-(y)+\left( e^{2\pi i \left({-\frac{1}{2}a-\frac{3}{4}}\right)} -1\right)2^{-a}e^{\pi i\left({\frac{1}{2}a-\frac{3}{4}}\right)}I_-(y),
\end{equation}
where now we use \eqref{Iminus} and obtain 
\begin{equation}\label{eq:proof of thm1: connection,eq 1.1}
	AC\,\tilde{u}_-(y)=-\sin(\pi a) \tilde{u}_-(y)+\left(-1+e^{2 \pi  i \left(-\frac{a}{2}-\frac{3}{4}\right)}\right) 2^{-a} e^{\pi  i \left(\frac{a}{2}-\frac{3}{4}\right)} \tilde{u}_+(y).
\end{equation}

Now, in order to obtain the first connection formula of~\eqref{connectionUV}, we divide the formula~\eqref{eq:proof of thm1: connection,eq 1.1} by~$\Gamma (\frac{1}{4}+\frac{a}{2})$ and, by using the Legendre duplication formula,  $\Gamma(z)\Gamma(z+\tfrac12)=2^{1-2z}\sqrt{\pi}\,\Gamma(2z)$, we have
\begin{equation}\label{eq:proof of thm1: connection,eq 2}
	\frac{\left(-1+e^{2 \pi  i \left(-\frac{a}{2}-\frac{3}{4}\right)}\right) 2^{-a} e^{\pi  i \left(\frac{a}{2}-\frac{3}{4}\right)} \Gamma\left(\frac{1}{2}-\frac{a}{2}\right)}{\Gamma\left(\frac{1}{4}+\frac{a}{2}\right)}= \sqrt{\frac{2}{\pi}}\frac{\pi}{\Gamma(\frac{1}{2}+a)}
\end{equation}
which shows that
$$AC\,u_-(y)=-\sin(\pi a) {u}_-(y)+ \sqrt{\frac{2}{\pi}}\frac{\pi}{\Gamma(\frac{1}{2}+a)}\, u_+(y)$$
which in view of~\eqref{connectionUV} and \eqref{vmU}, \eqref{eqconst} we obtain \eqref{vpV}.

\section{Proof of Theorem\,\ref{theorem2}}\label{PfTh2}

The construction of a fundamental system is similar to that in \S\ref{sec: PfTh1}, and somewhat simpler.

\subsection{Normalization}

A classical WKB calculation for equation~\eqref{mypcf} shows that possible behaviors at infinity are
\bel{asymEpm}
u(x)\sim   e^{\pm ix^2/4} \,(x^2)^{\mp \frac{ia}2-\frac14}\left(1+O(x^{-2})\right) \ \ \ \ \text {as } x\to +\infty
\ee
(up to multiplicative constants). 

Similarly to Section~\ref{sec: PfTh1}, we change variables in equation \eqref{mypcf} as
\begin{equation}
	s=x^2,\qquad  u(x)=v(x^2)
\end{equation}
and obtain 
\bel{eqvx2}
v''(s)+\frac1{2y}v'(s)+\left( {\frac{1}{16}}-{\frac {a}{4\,s}} \right) v \left( s \right) =0
\ee

\subsection{The functions~\texorpdfstring{$E_\pm$}{g}}

In this subsection we define the fundamental system~$E_+,E_-$ as solutions with the asymptotic behavior \eqref{asymEpm}.
 
First, we substitute 
\begin{equation}\label{subvplus}
	v(s)=e^{is/4}g_+(s)
\end{equation}
in~\eqref{eqvx2} and obtain that~$g_+$ solves
\begin{equation}
	g_+'' + \left( {
		\frac {i}{2}}+{\frac {1}{2~s}} \right) ~ g_+' + \left( \frac {i}{8s}-{\frac {a}{4~s
	}} \right) g _+ 
	=0.
\end{equation}
We use the Laplace transform
\begin{equation}
	g_+(s)=\int_{0}^{\infty }\!G_+ \left( p \right) {{\rm e}^{-sp}}\,{\rm d}p
\end{equation}
and then $G_+$ must satisfy
\begin{equation}
\left( -{\frac {3\,i}{4}}-{\frac {a}{2}}+3\,p \right) G_+ \left( p
\right) -p \left( i-2\,p \right) G_+ '\left( p
\right) 
=0
\end{equation}
therefore 
\bel{Gp}
G_+ \left( p\right)=p^{\frac{ia}2-\frac34}\, (1+2ip)^{-\frac{ia}2-\frac34},
\ee
modulo any multiplicative constant of the above. Therefore, we obtain
\bel{formg}
g_+(s)=\int_0^\infty\, e^{-ps} p^{\frac{ia}2-\frac34}\, (1+2ip)^{-\frac{ia}2-\frac34}\, dp,\ \ \ \Re s>0,
\ee
which is unique up to an arbitrary multiplicative constant.

We obtain another, independent, solution by substituting
\begin{equation}
	v(s)=e^{-is/4}g_-(s),
\end{equation}
which similarly gives
\bel{formh}
g_-(s)=\int_0^\infty\, e^{-ps} G_-(p)\,dp=\int_0^\infty\, e^{-ps} p^{-\frac{ia}2-\frac34}\, (1-2ip)^{\frac{ia}2-\frac34}\, dp,\ \ \ \Re s>0,
\ee
which is unique up to a multiplicative constant, where
\bel{Gm}
G_-(p)= p^{-\frac{ia}2-\frac34}\, (1-2ip)^{\frac{ia}2-\frac34}.
\ee

Formula \eqref{formg} is defined either as an integral, or by its Hadamard finite part for 
\begin{equation*}
	\tfrac{ia}2-\tfrac34\ne -1,-2,\ldots
\end{equation*}
and similarly \eqref{formh} is defined for
\begin{equation*}
-\tfrac{ia}2-\tfrac34\ne -1,-2,\ldots
\end{equation*}
(Once we choose suitable multiplicative constants the solutions obtained will have removable singularities at these excluded values, see below).

The two independent solutions of \eqref{eqvx2} obtained are
\bel{tildev}
\tilde{v}_\pm(s)=e^{\pm is/4}g_\pm(s)
\ee
which we normalize by defining\footnote{The branches below are so that  $-ip=e^{-\pi i/2}p$, consistent with \eqref{Uvmin}.}
\begin{equation}\label{solvpm}
	\begin{aligned}
		v_+(s)	&:=\frac1{\Gamma\left(\frac{ia}2+\frac14\right)}\,\tilde{v}_+(s)  = \frac1{\Gamma\left(\frac{ia}2+\frac14\right)}\,e^{is/4} \int_0^\infty\, e^{-ps} p^{\frac{ia}2-\frac34}\, (1+2ip)^{-\frac{ia}2-\frac34}\, dp
		,\ \ (\Re s>0)\\
		v_-(s)	&:=\frac1{\Gamma\left(-\frac{ia}2+\frac14\right)}\, \tilde{v}_-(s)=\frac1{\Gamma\left(-\frac{ia}2+\frac14\right)}\, e^{-is/4}\,\int_0^\infty\, e^{-ps} p^{-\frac{ia}2-\frac34}\, (1-2ip)^{\frac{ia}2-\frac34}\, dp, \ \ (\Re s>0)\\
	\end{aligned}
\end{equation}
 where the multiplying constants were chosen so that 
\begin{equation}\label{asymvpm}
	v_+(s)\sim   e^{is/4} \,s^{-\frac{ia}2-\frac14},\qquad v_-(s)\sim   e^{-is/4} \,s^{\frac{ia}2-\frac14},\qquad y\to+\infty,
\end{equation}
which is easily verified by  using Watson's Lemma.

 At the values
 \begin{equation*}
 	\pm\tfrac{ia}2-\tfrac34= -1,-2,\ldots
 \end{equation*}
 the Hadamard finite parts of the integrals have poles of order one, but the chosen pre-factors also have poles, which turns these excluded values into removable singularities of $v_\pm$, as proved by Lemma\,\ref{uminanalytic}.

Therefore, we identified the two linearly independent solutions of \eqref{mypcf} as
\begin{equation}
	E_\pm(a,x):=v_\pm(x^2)
\end{equation}
where $v_\pm$ are defined in \eqref{solvpm}; note that they have the asymptotic behavior~\eqref{asymEpmTh} which follows from \eqref{asymvpm}. The functions $E_\pm(a,x)$ are entire functions in $a$, after analytic continuation at the removable singularities.

\subsection{Analytic continuation and calculation of the connection formulas~\eqref{conEm},~\eqref{conEp}}

The solutions with asymptotic behavior \eqref{asymEpm} as $x\to -\infty$ are the functions $E_\pm(a,-x)$. To obtain the linear connection between the fundamental system at $+\infty$, namely
\begin{equation}
	E_+(a,x),\qquad E_-(a,x)
\end{equation}
and the fundamental system at~$-\infty$, namely  
 \begin{equation}
 	E_+(a,-x),\qquad E_-(a,-x),
 \end{equation}
 one can proceed by analytic continuation, as in Section~\ref{calccon}, with the upshot:

\begin{Proposition}\label{P1}
	
Let~$v_\pm$ be as in~\eqref{solvpm} and let~$\tilde{v}_\pm$ be as in~\eqref{tildev}. Then, the following hold

\begin{enumerate}[label=\alph*)]
	\item The analytic continuation by $2\pi$ of $\tilde{v}_-(s)$ is
	\bel{ACtv}
	AC\,\tilde{v}_-(s)=ie^{\pi a}  \tilde{v}_-(s)+\left(ie^{\pi a}+1\right)2^{ia}e^{-\pi i/4}\,  \tilde{v}_+(s).
	\ee
	\item  The following connection formulas hold 
	\begin{equation}\label{eq: P1, eq 1}
	AC\,{v}_-(s)=ie^{\pi a} {v}_-(s)+\frac{\sqrt{2\pi}}{\Gamma\left(\frac12-ia\right)} e^{\pi a/2}\,  {v}_+(s)
	\end{equation}
	and
	\begin{equation}\label{eq: P1, eq 2}
	AC\,{v}_+(s)=-ie^{\pi a} {v}_+(s)+\frac{\sqrt{2\pi}}{\Gamma\left(\frac12+ia\right)} e^{\pi a/2}\,  {v}_-(s)
	\end{equation}
	therefore the connection formulas~\eqref{conEm}, \eqref{conEp} hold.
	\item  For complex parameter $a$, formulas \eqref{connectionE} hold with $e^{i\phi_2}$ replaced by $\Gamma(\tfrac12+ia)\sqrt{\cosh \pi a}/\sqrt{\pi}$.
	\end{enumerate}

\end{Proposition}

\begin{proof}

a) As in Section \S\ref{calccon} we consider analytic continuation clockwise in $s$, which means we rotate the Laplace contour in \eqref{formh} counterclockwise. In this process we are left with a path hanging around the cut at the half line~$[-\tfrac i2,-i\infty)$\footnote{This is the line $-i \cdot (\frac{1}{2},+\infty)$}.

We start with \eqref{tildev}, \eqref{formh} and for $s>0$ we analytically continue by rotating $s$ clockwise, and simultaneously $p$ counterclockwise. After rotating~$p$ by an angle $\tfrac{3\pi}2 -0$ we obtain 
\begin{equation}
	\begin{aligned}
		AC\,\tilde{v}_-(s)	&=e^{-is/4}\int_{I_-}\, e^{-ps}G_-(p)\, dp \\
		&=e^{-is/4}\left[ \int_{I_+}+\left(\int_{I_-}-\int_{I_+}\right)\right]e^{-ps} G_-(p)\, dp\\
		&=e^{-is/4} \int_{I_+}e^{-ps}G_-(p)\, dp+e^{-is/4}\left[ e^{-2\pi i\left(\frac{ia}2-\frac34\right)}-1\right]\int_{-i/2}^{-i\infty}e^{-ps} G_-(p)\, dp,
	\end{aligned}
\end{equation}
where note the half-lines
\begin{equation}
	I_\pm:=e^{i\left(\frac{3\pi}2 \pm0\right)}\RR_+,
\end{equation}
and moreover note that for $p\in I_-$, $|p|>\tfrac12$ we have 
\begin{equation}
	\arg(1-2ip)=\pi,
\end{equation}
while for $p\in I_-$ we have
\begin{equation}
	\arg(1-2ip)=-\pi.
\end{equation}
We further continue up to $\arg s=-2\pi$, where in the last integral we change the integration variable to $p=-\tfrac i2+q$, and we obtain
\begin{equation}
	AC\,\tilde{v}_-(s)=e^{2\pi i\left(-\frac{ia}2-\frac34\right)}\tilde{v}_-(s) +\left[ e^{-2\pi i\left(\frac{ia}2-\frac34\right)}-1\right]2^{ia}e^{-\frac{i\pi}2\left(-\frac{ia}2-\frac34\right)}e^{-\frac{i\pi}2\left(\frac{ia}2-\frac34\right)} \tilde{v}_+(s) 
\end{equation}
which yields \eqref{ACtv}, since
\begin{equation}
	\left[ e^{-2\pi i\left(\frac{ia}2-\frac34\right)}-1\right]2^{ia}e^{-\frac{i\pi}2\left(-\frac{ia}2-\frac34\right)}e^{-\frac{i\pi}2\left(\frac{ia}2-\frac34\right)}= (ie^{\pi a}+1)2^{ia}e^{-i\pi/4}. 
\end{equation}

b) To obtain \eqref{eq: P1, eq 1} we use \eqref{solvpm} and the now established~\eqref{ACtv}, then simplify using  identities for the Gamma function (Euler's reflection and the Legendre duplication formulas) to obtain
\begin{equation*}
	\frac{\left(1+i e^{\pi  a}\right) 2^{i a}e^{-i\frac{1}{4}  \pi}  \Gamma \left(\frac{1}{4}+\frac{i a}{2}\right)}{\Gamma \left(\frac{1}{4}-\frac{i a}{2}\right)} = \frac{\sqrt{2 \pi } e^{\frac{\pi  a}{2}}}{\Gamma \left(\frac{1}{2}-i a\right)}.
\end{equation*}

Formula \eqref{eq: P1, eq 2} is obtained similarly.
(Note that, by symmetry, the formula for $AC\,{v}_+$ must be obtained from the formula for $AC\,{v}_-$ by replacing $i$ to $-i$, and $v_\pm$ by $v_\mp$.)

c) Follows using the fact that $\left| \Gamma\left(\tfrac12+ia\right)\right|^2=\tfrac{\pi}{\cosh \pi a}$ for real $a$. 
\end{proof}

\subsection{The proof of~\eqref{EEstar}. Relations between $E_\pm$ and $E,E^*$}\label{relationEpmEEstar}
A fundamental system of solutions for \eqref{mypcf} can also be found using \eqref{changevar} since
\begin{equation}
	U\left(ia,xe^{-\pi i/4}\right),\qquad U\left(-ia,xe^{\pi i/4}\right)
\end{equation}
solve \eqref{mypcf}.

To obtain integral formulas for them, we note that up to a numerical factor the function
\begin{equation}
	U\left(ia,xe^{-\pi i/4}\right)
\end{equation}
becomes 
\begin{equation}
	u_-(ia,y),\qquad y=-ix^2\in -i\RR_+. 
\end{equation}
To obtain an integral representation formula one must first analytically continue \eqref{vminus}, which amounts to rotating the path of integration by $\tfrac\pi2$, obtaining, for $y=-ix^2$, the following
\begin{equation}\label{Uvplus}
	\begin{aligned}
		U\left(ia,xe^{-\pi i/4}\right)&	=\frac1{\Gamma\left(\frac{ia}2+\frac14\right)}\,  \, \int_{e^{\frac{\pi i}2}\RR_+}\, e^{-px^2}\, p^{\frac{i}{2}a-\frac{3}{4}}\,\left(1+2p\right)^{-\frac{i}{2}a-\frac{3}{4}}\,\mathrm{d}p\\
		&=e^{-\frac{\pi a}4} e^{\frac{\pi i}8} v_+(x^2)\\
		&	=e^{-\frac{\pi a}4} e^{\frac{\pi i}8}E_+(a,x).
	\end{aligned}
\end{equation}

Similarly, the function
\begin{equation}
	U\left(-ia,xe^{\pi i/4}\right)
\end{equation}
is obtained from
\begin{equation}
	u_-(-ia,y),\qquad y=ix^2\in i\RR_+.
\end{equation}
To obtain an integral representaiton formula one must rotate the path of integration in \eqref{vminus} by $-\tfrac\pi2$, which yields
\begin{equation}\label{Uvmin}
	\begin{aligned}
		U\left(-ia,xe^{\pi i/4}\right)&	=\frac1{\Gamma\left(-\frac{ia}2+\frac14\right)}\, \int_{e^{-\frac{\pi i}2}\RR_+}\, e^{-px^2}\, p^{\frac{i}{2}a-\frac{3}{4}}\,\left(1+2p\right)^{-\frac{i}{2}a-\frac{3}{4}}\,\mathrm{d}p \\
		&= e^{-\frac{\pi a}4}\, e^{-\frac{\pi i}8}v_-(x^2)\\
		&	= e^{-\frac{\pi a}4}\, e^{-\frac{\pi i}8}E_-(a,x).
	\end{aligned}
\end{equation}

Combining \eqref{Uvplus}, \eqref{Uvmin} with \eqref{EU} we see that $E_+,E_-$ are , respectively appropriate constant multiples of $E,E^*$, and we obtain~\eqref{EEstar}.

\bibliographystyle{plain}
\bibliography{ParabolicCylFunc1}

\end{document}